\documentclass{jT}
\usepackage{amsmath,amssymb,amsbsy,amsfonts,amsthm,latexsym,
                        amsopn,amstext,amsxtra,euscript,amscd,mathrsfs,color,bm}
                        
\usepackage{float} 
\usepackage[english]{babel}
\usepackage{url}
\usepackage[colorlinks,linkcolor=blue,anchorcolor=blue,citecolor=blue,backref=page]{hyperref}
\usepackage[norefs,nocites]{refcheck}

 \usepackage[np]{numprint}

\npdecimalsign{\ensuremath{.}}

\usepackage{bibentry}

\usepackage[english]{babel}
\usepackage{mathtools}
\usepackage{todonotes}
\usepackage{url}
\usepackage[colorlinks,linkcolor=blue,anchorcolor=blue,citecolor=blue,backref=page]{hyperref}

\begin{document}

\newcommand{\commA}[2][]{\todo[#1,color=yellow]{A: #2}}
\newcommand{\commI}[2][]{\todo[#1,color=green!60]{I: #2}}
    
\newtheorem{theorem}{Theorem}
\newtheorem{lemma}[theorem]{Lemma}
\newtheorem{example}[theorem]{Example}
\newtheorem{algol}{Algorithm}
\newtheorem{corollary}[theorem]{Corollary}
\newtheorem{prop}[theorem]{Proposition}
\newtheorem{definition}[theorem]{Definition}
\newtheorem{question}[theorem]{Question}
\newtheorem{problem}[theorem]{Problem}
\newtheorem{remark}[theorem]{Remark}
\newtheorem{conjecture}[theorem]{Conjecture}

\newtheorem*{theorem*}{Theorem}

\def\xxx{\vskip5pt\hrule\vskip5pt}

\def\Cmt#1{\underline{{\sl Comments:}} {\it{#1}}}

\newcommand{\Modp}[1]{
\begin{color}{blue}
 #1\end{color}}
 
 \def\bl#1{\begin{color}{blue}#1\end{color}} 
 \def\red#1{\begin{color}{red}#1\end{color}} 

%\newcommand{\eqname}[1]{\tag{#1}}% Tag equation with name

%%%%%%%%%%%%%%%%%%%%%%%%%
% Alphabet calligraphic %
%%%%%%%%%%%%%%%%%%%%%%%%%
\def\cA{{\mathcal A}}
\def\cB{{\mathcal B}}
\def\cC{{\mathcal C}}
\def\cD{{\mathcal D}}
\def\cE{{\mathcal E}}
\def\cF{{\mathcal F}}
\def\cG{{\mathcal G}}
\def\cH{{\mathcal H}}
\def\cI{{\mathcal I}}
\def\cJ{{\mathcal J}}
\def\cK{{\mathcal K}}
\def\cL{{\mathcal L}}
\def\cM{{\mathcal M}}
\def\cN{{\mathcal N}}
\def\cO{{\mathcal O}}
\def\cP{{\mathcal P}}
\def\cQ{{\mathcal Q}}
\def\cR{{\mathcal R}}
\def\cS{{\mathcal S}}
\def\cT{{\mathcal T}}
\def\cU{{\mathcal U}}
\def\cV{{\mathcal V}}
\def\cW{{\mathcal W}}
\def\cX{{\mathcal X}}
\def\cY{{\mathcal Y}}
\def\cZ{{\mathcal Z}}

\def\C{\mathbb{C}}
\def\F{\mathbb{F}}
\def\K{\mathbb{K}}
\def\L{\mathbb{L}}
\def\G{\mathbb{G}}
\def\Z{\mathbb{Z}}
\def\R{\mathbb{R}}
\def\Q{\mathbb{Q}}
\def\N{\mathbb{N}}
\def\M{\textsf{M}}
\def\U{\mathbb{U}}
\def\P{\mathbb{P}}
\def\A{\mathbb{A}}
\def\fp{\mathfrak{p}}
\def\n{\mathfrak{n}}
\def\X{\mathcal{X}}
\def\x{\textrm{\bf x}}
\def\w{\textrm{\bf w}}
\def\a{\textrm{\bf a}}
\def\k{\textrm{\bf k}}
\def\ee{\textrm{\bf e}}
\def\ovQ{\overline{\Q}}
\def \Kab{\K^{\mathrm{ab}}}
\def \Qab{\Q^{\mathrm{ab}}}
\def \Qtr{\Q^{\mathrm{tr}}}
\def \Kc{\K^{\mathrm{c}}}
\def \Qc{\Q^{\mathrm{c}}}
\newcommand \rank{\operatorname{rk}}
\def\ZK{\Z_\K}
\def\ZKS{\Z_{\K,\cS}}
\def\ZKSf{\Z_{\K,\cS_f}}
\def\ZKSfG{\Z_{\K,\cS_{f,\Gamma}}}

\def\bF{\mathbf {F}}

\def\({\left(}
\def\){\right)}
\def\[{\left[}
\def\]{\right]}
\def\<{\langle}
\def\>{\rangle}

\def\gen#1{{\left\langle#1\right\rangle}}
\def\genp#1{{\left\langle#1\right\rangle}_p}
\def\genPs{{\left\langle P_1, \ldots, P_s\right\rangle}}
\def\genPsp{{\left\langle P_1, \ldots, P_s\right\rangle}_p}

\def\e{e}

\def\eq{\e_q}
\def\fh{{\mathfrak h}}

\def\lcm{{\mathrm{lcm}}\,}

\def\({\left(}
\def\){\right)}
\def\fl#1{\left\lfloor#1\right\rfloor}
\def\rf#1{\left\lceil#1\right\rceil}
\def\mand{\qquad\mbox{and}\qquad}

\def\jt{\tilde\jmath}
\def\ellmax{\ell_{\rm max}}
\def\llog{\log\log}

\def\m{{\rm m}}
\def\ch{\hat{h}}
\def\GL{{\rm GL}}
\def\Orb{\mathrm{Orb}}
\def\Per{\mathrm{Per}}
\def\Preper{\mathrm{Preper}}
\def \S{\mathcal{S}}
\def\vec#1{\mathbf{#1}}
\def\ov#1{{\overline{#1}}}
\def\Gal{{\mathrm Gal}}
\def\Sp{{\mathrm S}}
\def\tors{\mathrm{tors}}
\def\PGL{\mathrm{PGL}}
\def\wH{{\rm H}}
\def\Gm{\G_{\rm m}}

\def\house#1{{%
    \setbox0=\hbox{$#1$}
    \vrule height \dimexpr\ht0+1.4pt width .5pt depth \dp0\relax
    \vrule height \dimexpr\ht0+1.4pt width \dimexpr\wd0+2pt depth \dimexpr-\ht0-1pt\relax
    \llap{$#1$\kern1pt}
    \vrule height \dimexpr\ht0+1.4pt width .5pt depth \dp0\relax}}

\newcommand{\bfalpha}{{\boldsymbol{\alpha}}}
\newcommand{\bfomega}{{\boldsymbol{\omega}}}

\newcommand{\Ch}{{\operatorname{Ch}}}
\newcommand{\Elim}{{\operatorname{Elim}}}
\newcommand{\proj}{{\operatorname{proj}}}
\newcommand{\h}{{\operatorname{\mathrm{h}}}}
\newcommand{\ord}{\operatorname{ord}}

\newcommand{\hh}{\mathrm{h}}
\newcommand{\aff}{\mathrm{aff}}
\newcommand{\Spec}{{\operatorname{Spec}}}
\newcommand{\Res}{{\operatorname{Res}}}

\def\fA{{\mathfrak A}}
\def\fB{{\mathfrak B}}

\numberwithin{equation}{section}
\numberwithin{theorem}{section}

\title[On abelian points in subvarieties of a torus]{On abelian points of varieties intersecting subgroups in a torus}

\author[J. Mello] {Jorge Mello}
\address{Max Planck Institute for Mathematics. Vivatsgasse 7, 53111, Bonn, Germany.}
\email{jmelloguitar@gmail.com}



\maketitle

\begin{abstr} We show, under some natural conditions, that the set of abelian points on the non-anomalous dense subset of a closed irreducible subvariety $X$ intersected with the union of connected algebraic subgroups of codimension at least $\dim X$ in a torus is finite, generalising results of Ostafe, Sha, Shparlinski and Zannier (2017). We also generalise their structure theorem for such sets when the algebraic subgroups are not necessarily connected, and obtain a related result in the context of curves and arithmetic dynamics.
\end{abstr}

\section{Introduction}
Given  non-zero complex numbers $\alpha_1, \ldots, \alpha_n \in \C^*$, we say that they are \textit{multiplicatively dependent} if there exist integers $k_1,\ldots,k_n \in \Z$, not all zero, such that 
\begin{equation*} 
\alpha_1^{k_1}\cdots \alpha_n^{k_n} = 1.
\end{equation*} If moreover $k_1,...,k_n$ are relatively prime, then $\alpha_1, \ldots, \alpha_n$ are said to be \textit{primitively dependent}.

Multiplicative dependence of algebraic numbers has been studied for a long time.
For examples, see \cite{ BOSS, BMZ, LvdP, vdPL, OSSZ1, OSSZ2}.

Multiplicative dependence arises naturally from the study of points on subvarieties of tori. 
Namely, for  $\G_m$ the multiplicative algebraic group over the complex numbers $\C$, endowed with the multiplicative group law,
there are many works investigating the intersection of an algebraic variety in $\G_m^n$ and 
the algebraic  subgroups in $\G_m^n$; for example \cite{BMZ1, BMZ, Hab, OSSZ1}.

In \cite{OSSZ1}, it was studied the multiplicative dependence of coordinates of curves over fields with the Bogomolov property (see Definition \ref{Bog}), for example the abelian closure of a number field. Namely, its authors proved the following 

\begin{theorem*} \cite[Theorem 2.10]{OSSZ1}  \textit{ Let} $X \subset G= \mathbb{G}_m^n$ \textit{be an irreducible curve defined over a number field} $K$ \textit{ and not contained in any translate of a proper algebraic subgroup of } $\mathbb{G}_m^n$. \textit{Suppose that} $L \supset K$ \textit{has the Bogomolov Property. Then there are at most finitely many points in $X(L)$ whose coordinates are primitively dependent.}
\end{theorem*}

For multiplicative dependence of values of rational functions in this context see \cite{OSSZ2}.
In \cite{OSSZ2} it was also studied the multiplicative dependence of elements in an orbit of an algebraic dynamical system, 
and recently in \cite{BOSS} this was extended to the more general setting of multiplicative dependence modulo a finitely generated multiplicative group.

In this paper, we want to study the extension of \cite[Theorem 2.10]{OSSZ1} stated above to higher dimensional varieties, as suggested in \cite[Remark 4.1]{OSSZ1}. The statements studied and obtained require the definition of the open \textit{anomalous set} $X^{oa}$ of an irreducible closed subvariety $X \subset \mathbb{G}_m^n$, to be defined properly further down (Definition \ref{defanom}). If $C$ is a curve then $C^{oa}=C$ if and only if $C$ is not contained in a translate of a proper algebraic subgroup of $\mathbb{G}_m^n$. Among other results, we prove the following

 \begin{theorem} Let $X \subset G= \mathbb{G}_m^n$be an irreducible closed subvariety defined over a number field $K$. Suppose that $L \supset K$ has the Bogomolov Property. Then there are at most finitely many points in $X^{oa}(L)$ contained in some connected algebraic subgroup of $G$ of codimension at least $\dim X$\end{theorem}

In Section 2 we state preliminar notation, definitions, and results. In Section 3 we start proving two main results, including a structure theorem for abelian points for the non-anomalous dense subset of a variety that generalises \cite[Theorem 2.1]{OSSZ1}.  Finally, we consider similar problems in the context of arithmetic dynamics, in particular when curves intersect periodic hypersurfaces.
\section{Preliminaries}
\subsection{Notation} 
We denote by $\mathbb{U}$ the set of roots of unity in $\overline{\mathbb{Q}}$.
For a field $K$, we use $\ov K$ to denote the algebraic closure of $K$, $K_c$ to denote the cyclotomic closure of $K$, and $K_{ab}$ to denote the abelian closure of $K$. We use $h$ to denote the usual logarithmic Weil height on $\overline{\mathbb{Q}}$.

As usual, for given real valued functions  $U$ and $V$, the notations $U\ll V$, $V\gg  U$ and
$U=O(V)$ are all equivalent to the statement that the inequality
$|U|\le c V$ holds with some constant $c>0$.
\subsection{Preliminar definitions and results} In order to prove our desired extension results on finiteness of abelian points, we start recalling the general concept of anomalous varieties. Namely, let $Y \subset X \subset \mathbb{G}_m^n$ be algebraic subvarieties of a multiplicative torus, $Y$ irreducible, and let $s$ be a non-negative integer.
Then $Y$ is said $s$-\textit{anomalous} (for $X$) if $\dim Y \geq 1$ and there exists a coset (translate of a subgroup) $ H \subset \mathbb{G}_m^n$ satisfying
$$
Y \subset H \text{ and } \dim Y > s + \dim H - n.
$$
\begin{definition}\label{defanom}
The non-anomalous $s$-part of $X$ is
$$
X^{oa,s}= X \setminus (\text{ union of all }s\textit{-anomalous subvarieties }),
$$ and we denote $X^{oa}=X^{oa,\dim X}$. 
\end{definition} In \cite{BMZ1}, it was proved the density of such set.
\begin{lemma}\cite{BMZ1}
$X^{oa}$ is a Zariski open subset of $X$.
\end{lemma} We aim to intersect varieties with algebraic subgroups of the torus, thus we use the following notation
\begin{definition}\label{defcosetunion} Let $d$ be a non-negative integer. We denote the union of all algebraic subgroups of $\mathbb{G}_m^n$ of codimension at least $d$ ( or dimension at most $n-d$) by
$$\mathcal{G}^{[d]}= \bigcup_{\text{codim } H \geq d} H(\bar{\mathbb{Q}}).$$
\end{definition} The following result was in part initially conjectured by Bombieri, Masser and Zannier \cite{BMZ1} and proved by Habegger (2009)\cite{Hab},  together with the subsequent Lemma. They play a crucial role in the verification of the present conclusions.
\begin{lemma}\cite[Bounded Height Theorem]{Hab}\label{boundedheight}
Let $G= \mathbb{G}_m^n$ and let $X \subset G$ be an irreducible closed subvariety defined over $\bar{\mathbb{Q}}$. Then the height is bounded from above on $X^{oa,s}\cap \mathcal{G}^{[s]}$.\end{lemma}
\begin{lemma}\cite{Hab}\label{finiteness} Under the conditions of Lemma \ref{boundedheight}, 
$$X^{oa}\cap \mathcal{G}^{[\dim X +1]}$$ is finite.
\end{lemma} 
 In order to work with connected algebraic subgroups, the following special concept for lattices and its properties will be useful.
\begin{definition}
We say that a free lattice $\mathcal{L}= \sum_{j=1}^r \mathbb{Z} (a_{1j},...,a_{nj})$ in $ \mathbb{Z}^n$ is \textit{primitive} if the gcd of the $r\times r$-minors of the matrix $(a_{ij})_{i,j}$ is 1, or equivalently if $\mathbb{Q}\mathcal{L} \cap \mathbb{Z}^n=\mathcal{L}$, according to \cite[Proposition 4.2]{Z}.
\end{definition}
In fact, it turns out that primitive lattices characterize connected algebraic subgroups.
\begin{lemma}\cite[Corollary 4.5]{Z}\label{lemprim}
Let $H$ be an algebraic subgroup of $\mathbb{G}_m^n$. Then there exists a free lattice  $\mathcal{L}= \sum_{j=1}^r \mathbb{Z} (a_{1j},...,a_{nj})$ such that $H$ is defined by the equations
$$x_1^{a_{1j}}...x_n^{a_{nj}}=1, \qquad j=1,...,r,$$ and $H$ is connected if and only if $\mathcal{L}$ is primitive.
\end{lemma}
In higher dimensions, we will also make use of a way to find small bases for integer lattices as stated below.
\begin{lemma}
\label{lem:HB}\cite[Lemma 1]{HB}
Let $\cL\subseteq \Z^d$ be a lattice of rank $m$.  Then $\cL$ has a basis $\vec{b}_1,\ldots,\vec{b}_m$ such that, for each $\vec{x}\in \cL$, we may write 
$$\vec{x}=\sum_{j=1}^{m}\lambda_j\vec{b}_j,$$
with 
$$\lambda_j\ll \frac{\|\vec{x}\|}{\|\vec{b}_j\|}.$$
We also have 
$$\det \cL\ll \prod_{j=1}^{m}\|\vec{b}_i\|\ll \det \cL.$$ 
\end{lemma}
 In our context, it is natural to consider the following modification of Definition \ref{defcosetunion}:
\begin{definition} Let $d$ be a non-negative integer. We define
$$\widetilde{\mathcal{G}}^{[d]}= \bigcup_{\text{codim } H \geq d} H(\bar{\mathbb{Q}}),$$
where the union runs over all connected algebraic subgroups $H \subset \mathbb{G}_m^n$ of codimension at least $d$.
\end{definition} We also state the so called Bogomolov property for fields that are the object of our work.
\begin{definition}\label{Bog} \cite{ADZ} We say that a subfield $L$ of $\overline{\mathbb{Q}}$ has the Bogomolov property  if there exists a constant $C(L)$ which depends only on $L$, such that for any $\alpha \in L^* \setminus \mathbb{U}$ we have $h(\alpha) \geq C(L)$.
\end{definition}
 \section{Proofs of the main results} Now we generalize \cite[Theorem 2.10]{OSSZ1} for higher dimension varieties.
\begin{theorem}\label{mainthm}
Let $X \subset G= \mathbb{G}_m^n$ be an irreducible closed subvariety defined over a number field $K$. Suppose that $L \supset K$ has the Bogomolov property. Then  $$X^{oa}\cap \widetilde{\mathcal{G}}^{[\dim X]}(L)$$ is finite. \end{theorem}
\begin{proof}
Let $s=\dim X$ and $P \in X^{oa,s}\cap \widetilde{\mathcal{G}}^{[s]}(L)$. Then the coordinates $\xi_i=x_i(P)$ generate a multiplicative subgroup $\Gamma_P \subset (L)^*$ of rank $ r \leq n -s$. By elementary abelian group decomposition, we can write
\begin{equation}\label{eq1.1}
\xi_i=\zeta_i \prod_{j=1}^r g_j^{m_{ij}}, \quad \quad i=1,...,n,
\end{equation} for generators $g_j \in L$ of the torsion-free part of $\Gamma_P$, integers $m_{ij}$ and roots of unity $\zeta_i \in L$.

It follows from a result of Schlickewei (  see \cite[Lemma 2]{BMZ}) that one can find generators $g_i$ so that , for any integers $b_1,...,b_r$, 
\begin{equation}
h(g_1^{b_1}...g_r^{b_r}) \geq c_r(|b_1| h(g_1)+...+|b_r| h(g_r)),
\end{equation}where $c_r>0$ depends only on $r$. By Lemma \ref{boundedheight}, the height of points in $X^{oa,s}\cap \mathcal{G}^{[s]}(\bar{\mathbb{Q}})$ is uniformly bounded, so that
 \begin{equation}
h \left(  \prod_{i=1}^r g_j^{m_{ij}}\right) \ll 1,
\end{equation} and the implicit constant depends only on $X$ and $n$.

By the Bogomolov property of $L$, we have 
\begin{equation}
h(g_j) \gg 1,
\end{equation} where the implicit constant depends only on $L$.
Combining (3.3) and (3.4), we obtain that the absolute values $|m_{ij}|$ are upper bounded independently of the point $P \in X^{oa,s}\cap \mathcal{G}^{[s]}(L)$. 

For instance, for such a point $P$ in $X^{oa,s}\cap \widetilde{\mathcal{G}}^{[s]}(L)$,  we recall that
$$
 P \in X^{oa,s}\cap  \bigcup_{\text{codim } H \geq s} H(L),
$$ and the union is taken over all connected algebraic subgroups $H \subset \mathbb{G}_m^n$ of codimension at least $s$. Since $X^{oa,s}= \emptyset$ for $s< \dim X$, we can suppose that $s \geq \dim X$. In this case, there exists a connected subgroup $H \subset G$ of codimension $d \geq s$ such that $P \in X^{oa,s}\cap H(L)$. We may write $H$ as defined by the equations
$$x_1^{a_{1j}}...x_n^{a_{nj}}=1, \qquad j=1,...,d.$$
 The rank $r$ of (3.1) is at most $n-d$, the linearly independent integer valued vectors $$(a_{1j},...,a_{nj}), \quad j=1,...,d$$ are orthogonal to the vectors $(m_{1j},...,m_{nj}), j=1,...,r$ and we claim that the $a_{ij}$'s can be chosen  so that they are also bounded independently of $P$. In fact, denoting the lattice generated by the vectors $(m_{1j},...,m_{nj}), j=1,...,r$ by $\mathcal{M}$, and making $\mathcal{L}=\sum_{j=1}^d \mathbb{Z} (a_{1j},...,a_{nj})$, the orthogonal set $$\mathcal{L}^\bot=\{u \in \mathbb{Z}^n : u.\ell=0 \text{ for all } \ell \in \mathcal{L}  \}$$ is also an integer primitive lattice, according to \cite[Section 2]{NS}(there the word complete is used instead of primitive).
  Since we can assume that $|a_{ij}|\ll \det(\mathcal{L})$ by Lemma 2.8, where the bound depends only on the dimension of the ambient space according to \cite[eq. (14)]{D}, and the subgroup $H$ is connected, the lattice $\mathcal{L}=\sum_{j=1}^d \mathbb{Z} (a_{1j},...,a_{nj})$ is primitive by Lemma 2.8  and so $\det(\mathcal{L})=\det(\mathcal{L}^{\perp})$ by \cite[Thm. 1]{NS}, where $\mathcal{L}^{\perp}$ is the orthogonal lattice to $\mathcal{L}$. The vectors of the integer lattice $\mathcal{M}$ spanned by the vectors $(m_{1j},...,m_{nj})$ are orthogonal to the vectors of $\mathcal{L}$, and thus $\mathcal{M} \subset \mathcal{L}^{\perp}$. If rank $(\mathcal{M})=r \leq n-d-1\leq n-\dim X -1,$ then the point $P$  is in $X^{oa,s}\cap \mathcal{G}^{[\dim X+1]} = X^{oa}\cap \mathcal{G}^{[\dim X+1]}$ which is finite by Lemma 2.5. Thus, we can assume that $r=n-d$. In this case, $\mathcal{M} \subset \mathcal{L}^{\perp}$ and both have the same rank. This implies $\det (\mathcal{L}^\perp)\leq \det(\mathcal{M})$, and therefore
$$
|a_{ij}|\ll \det(\mathcal{L})=\det(\mathcal{L}^{\perp})\leq \det(\mathcal{M}),
$$ which proves that the $a_{ij}$'s are bounded independently of $P$, because the $m_{ij}$'s are, and so is $\det(\mathcal{M})$, proving the claim. 
 
 Now we claim that $\dim (X \cap H) =0$ and thus $X \cap H$ is finite. In fact, if $\dim (X \cap H) \geq 1$, then for $Y $ the irreducible component of $ X \cap H \subset H$ containing $P$, we have that $$\dim Y  \geq 1 >0 \geq  s + \dim H -n=s-d.$$ Thus, $P \in X^{oa,s} \cap Y$ with $Y$ $s$-anomalous for $X$, which is a contradiction with the definition of $X^{oa,s}$. Thus, $X \cap H$ is finite. Since the $|a_{ij}|$ are bounded, $P$ belongs to a finite set.
 \end{proof}Since the abelian closure of a number field has the Bogomolov property (this was proved by Amoroso-Zannier\cite{AZ}), we deduce the following \begin{corollary}
Under the conditions of Theorem \ref{mainthm}, $$X^{oa}\cap \widetilde{\mathcal{G}}^{[\dim X]}(K_{ab})$$ is finite. 
\end{corollary} If we enlarge subgroups multiplying them by a finitely generated group, we can prove a related  result if we cut sets modulo torsion, if we assume that $X$ is \textit{non-degenerate} (\cite[Def. 1.1]{M}). We write $x \sim y$ in $\mathbb{G}_m^n(\overline{\mathbb{Q}})$ if $x=uy$ with $u \in \mathbb{U}^n.$ Also, let $\Gamma_\epsilon$ be the subset of $\Gamma$ of elements having height at most $\epsilon$ and for $S \subset G$ define $\mathcal{B}(S,\epsilon)=\{xy \in G : x \in S, y \in G \text{ with } h(y)\leq \epsilon \}.$
 \begin{theorem}
Under the same conditions of Thm. 3.1, assume that $X$ is non-degenerate, that $L/K$ has the Bogomolov property, and that $\Gamma$ is a subgroup of finite rank of $G(L)$. Then $$
(X^{oa}\cap \Gamma \cdot \tilde{\mathcal{G}}^{[\dim X + 1]}(L))/\sim \textit{ is finite.} 
$$
\end{theorem}
\begin{proof}
Let $s=\dim X +1$. Note that for any given $\epsilon \geq 0,$ one has that $ \Gamma_\epsilon/\sim$ is finite. In fact, as in the proof of Theorem 3.1, the elements of $\Gamma_\epsilon$ have height bounded by $\epsilon$, and the lemma of Schlickewei used in the proof of Thm. 3.1 together with the Bogomolov property imply a lower bound for elements in $\Gamma$, which together show that the exponents of given generators of $\Gamma$ representing any element of $\Gamma_\epsilon$ are uniformly bounded. Thus, $ \Gamma_\epsilon/\sim$ is finite. Hence,
$$
(X^{oa}\cap \Gamma_{\epsilon} \cdot \tilde{\mathcal{G}}^{[s]}(L))/\sim  \text{ } \subset \left(\bigcup_{\gamma \in \Gamma_\epsilon}\gamma\cdot((\gamma^{-1}X)^{oa}\cap \cdot \tilde{\mathcal{G}}^{[s]}(L))\right)/\sim
$$ is finite by Lemma 2.5 and the finiteness of $\Gamma_\epsilon/\sim.$ Moreover, by Maurin's \cite[Thm. 11.8]{M} and the arguments in its proof applied to $\mathcal{B}(\tilde{\mathcal{G}}^{[s]}(L),\epsilon)$ in place of $\mathcal{C}(\tilde{\mathcal{G}}^{[s]}(L),\epsilon)$, there exist $\gamma_1,..,\gamma_N \in \Gamma$ and $\epsilon_1>0$ so that
$$
X^{oa}\cap \Gamma \cdot \tilde{\mathcal{G}}^{[s]}(L) \subset \bigcup_{i=1}^N X^{oa}\cap \gamma_i \cdot \mathcal{B}(\tilde{\mathcal{G}}^{[s]}(L),\epsilon_1).
$$Intersecting both sides of the inclusion above with $\Gamma \cdot \tilde{\mathcal{G}}^{[s]}(L)$ taking quotients by $\sim$ we obtain
$$
X^{oa}\cap \Gamma \cdot \tilde{\mathcal{G}}^{[s]}(L)/\sim \text{ }\subset \bigcup_{i=1}^N X^{oa}\cap \gamma_i \cdot (\Gamma_{\epsilon_1}\cdot \tilde{\mathcal{G}}^{[s]}(L))/\sim\text{ }=\bigcup_{i=1}^N \gamma_i \cdot \left((\gamma_i^{-1}X)^{oa}\cap \Gamma_{\epsilon_1}\cdot \tilde{\mathcal{G}}^{[s]}(L)) \right)/\sim,
$$ and the last set is the union of finite sets again.
\end{proof}

 \subsection{A structure theorem over $K_{ab}$}
 Here we generalise \cite[Theorem 2.1]{OSSZ1} to higher dimensional varieties.
 
 \begin{definition}We say that a $d$-uple $$\phi:(x_1,...,x_n) \mapsto (x_1^{a_{11}}...x_n^{a_{n1}},...,x_1^{a_{1d}}...x_n^{a_{nd}})$$ of characters (called a $d$-\textit{character}) is \textit{primitive} if the lattice $\sum_{j=1}^d \mathbb{Z} (a_{1j},...,a_{nj})$ in $ \mathbb{Z}^n$ is primitive. For $X \subset \mathbb{G}_m^n$ an algebraic variety, we denote by $\phi_X$ the restriction of $\phi$ to $X$.\end{definition}
 \begin{definition} Let $X \subset \mathbb{G}_m^n$ be an algebraic variety of dimension $\dim X=d$.
 We define $\Phi_X^{[d]}$ to be the set of primitive $d$-characters $\varphi=(\phi_1,...,\phi_{d}):\mathbb{G}_m^n \rightarrow \mathbb{G}_m^{d}$ with the property that there exists a birational map $\rho: \mathbb{G}_m^{d} \rightarrow X$ such that $$(\varphi_X \circ \rho) (t_1,...,t_{d})= (t_1^{m_1},...,t_{d}^{m_{d}})$$ for integers $m_1,...,m_{d}$.
 
 We use $\Phi_{X,c}^{[d]}$ to denote the subset of $\Phi_X^{[d]}$ of characters for which $\rho$ is defined over $K_c$.
 \end{definition}

\begin{theorem}\label{structural}
Let $X \subset \mathbb{G}_m^n$ be an irreducible closed subvariety of dimension $d$ defined over a number field $K$. If $X$ is not $d$-anomalous, then $\Phi^{[d]}_X$ is finite. Moreover, there exists a finite union $W$ of proper cosets of $\mathbb{G}_m^d$ such that $X^{oa}\cap \mathcal{G}^{[d]}(K_{ab})$ is contained in the union of $\bigcup_{\phi \in \Phi_{X,c}^{[d]}}\phi_X^{-1}(\mathbb{U}^d)$, a finite set, and sets $\rho^{-1}(W\cap \mathbb{U}^d)$ where $\rho:X \rightarrow  \mathbb{G}_m^d$ belongs to a prescribed finite set of rational maps.
\end{theorem}

\begin{proof}
Let $\widetilde{X}$ be a smooth projective model defined over $K$ and $K$-birational to $X$, which exists by Hironaka's resolution of singularities \cite{Hi}. Given a $d$-character $\varphi=(\varphi_1,...,\varphi_d)$, its restriction $\varphi_X=(\varphi_{1X},...,\varphi_{dX})$ to $X$ is a $d$-uple of rational functions on $X$, and also on $\widetilde{X}$, so we may consider $(\text{div}(\varphi_{1X}),...,\text{div}(\varphi_{dX}))$ inside $(\text{Div}(\widetilde{X}))^d$. We then have a homomorphism defined by $$\varphi \mapsto (\text{div}(\varphi_{1X}),...,\text{div}(\varphi_{dX})) \in (\text{Div}(\widetilde{X}))^d.$$ We note that this map is injective, since any map in the kernel yields $\varphi_{iX}$ constant for each $i$, and thus $X$ itself would be $d$-anomalous, which cannot happen.

Let $\varphi \in \Phi_X^{[d]}$, so we may write $\varphi_X \circ \rho(t_1,...,t_d)=(t_1^{m_1},...,t_d^{m_d})$, where $\rho: \mathbb{G}_m^d \rightarrow X$ is a birational isomorphism and $(t_1,...,t_d)$ are coordinates on $\mathbb{G}_m^d$. We may extend $\rho$ to a birational map $\rho: (\mathbb{P}_1)^d \rightarrow \widetilde{X}$, and the same equation holds on viewing $\varphi_X$ as a map from $\widetilde{X}$ to $(\mathbb{P}_1)^d$. Hence $\varphi_{iX}$ can only have one zero and one pole, and so its divisor is of the shape $M((Y_i)-(Z_i))$ for each $i$. And since $\varphi_{iX}$ is a monomial in the coordinates, $Y_i,Z_i$ necessarily lie among the zeros and poles of the coordinate functions $x_1,...,x_n$ viewed as functions on $\widetilde{X}$, and so there are only finitely many possibilities for the prime divisors $Y_i, Z_i$.

Finally, if $\varphi=(\varphi_1,...,\varphi_d)$ and $ \psi=(\psi_1,...,\psi_d)$ lie in $ \Phi_X^{[d]}, j \in \mathbb{N}, Y_i, Z_i$ are prime divisors for each $i=1,...,j$ and $D_{j+1},...,D_d$ are divisors such that
$$
(\text{div}(\varphi_{1X}),...,\text{div}(\varphi_{dX}))=(M_1(Y_1-Z_1),...,M_j(Y_j-Z_j),D_{j+1},...,D_d)
$$ and
$$
(\text{div}(\psi_{1X}),...,\text{div}(\psi_{dX}))=(L_1(Y_1-Z_1),...,L_j(Y_j-Z_j),D_{j+1},...,D_d),
$$ then $$(\text{div}(\varphi^{L_1}_{1X}),...,\text{div}(\varphi^{L_j}_{jX}),...,\text{div}(\varphi_{dX}))= (\text{div}(\psi^{M_1}_{1X}),...,\text{div}(\psi^{M_j}_{jX}),...,\text{div}(\psi_{dX})),$$ and hence $$((\varphi^{L_1}_{1X}),...,(\varphi^{L_j}_{jX}),...,(\varphi_{dX}))= ((\psi^{M_1}_{1X}),...,(\psi^{M_j}_{jX}),...,(\psi_{dX}))$$ by injectivity. 
Both $\varphi, \psi$ are primitive, therefore $L_i=\pm M_i$ for each $i$. We can make an analogous construction  with the first $i$ coordinates replaced by other $i$ random coordinates. So each $d$-uple of pairs $((Y_1,Z_1),...,(Y_d,Z_d))$ can give rise to at most $2^d$ elements of $\Phi_X^{[d]}$, and the finiteness of $\Phi_X^{[d]}$ follows.

Let $P \in X^{oa}\cap {\mathcal{G}}^{[d]}(K_{ab})$, so the coordinates $\xi_i=x_i(P)$ generate a multiplicative subgroup $\Gamma_P \subset (K_{ab})^*$ of rank $ r \leq n -d$, and we can assume by Lemma \ref{finiteness} that $r=n-d$. By elementary abelian group decomposition again, we can write
\begin{equation}\label{eq1.1}
\xi_i=\zeta_i \prod_{j=1}^r g_j^{m_{ij}}, \quad \quad i=1,...,n,
\end{equation} for generators $g_i \in K_{ab}$ of the torsion-free part of $\Gamma_P$, integers $m_{ij}$ and roots of unity $\zeta_i \in K_{ab}$.

As in the proof of Theorem 3.1, we derive that
 \begin{equation}
h \left(  \prod_{i=1}^r g_j^{m_{ij}}\right) \ll 1,
\end{equation} and the implicit constant depends only on $X$ and $n$.

By the Bogomolov property of $K_{ab}$, we have 
\begin{equation}
h(g_j) \gg 1,
\end{equation} where the implicit constant depends only on $K$.
Combining (3.6) and (3.7), we obtain that the absolute values $|m_{ij}|$ are upper bounded independently of any point $P \in X^{oa}\cap \mathcal{G}^{[d]}(K_{ab})$. 
 For such, we can take
$$
(b_{1i},...,b_{ni}) \in \mathbb{Z}^n, i=1,...,d
$$ nonzero integer vectors forming a basis for the orthogonal vector space to the $(m_{1j},...,m_{nj})$, whose associate lattice we can suppose to be primitive as before. In this way, we see that we can obtain a primitive $d-character$ $\pi$ from a prescribed finite set, and for $P=(\xi_1,...,\xi_n)$ as above, we have that $\pi(P)=(\zeta_{P_1},...,\zeta_{P_n})$ where the $\zeta_{Pi}$'s are roots of unity. 
Therefore $\pi$ can be taken from  a prescribed finite set, and sends the point $P$ to $\mathbb{U}^d$. Thus we only need to treat the case in which infinitely many points in $X^{oa}\cap \mathcal{G}^{[d]}(K_{ab})$ are sent to $\mathbb{U}^d$ by $\pi$. The preimage of  an element in $\mathbb{U}^d$ by $\pi$ is a $d$-dimensional coset $H$. If $X\cap H$ was infinite, $X$ would be $d$-anomalous, hence we may assume that
$\pi_X(X^{oa}\cap \mathcal{G}^{[d]}(K_{ab})) \cap \mathbb{U}^d$ is infinite.

Let $T_k$ be the set of torsion points of $\mathbb{G}_m^k, k \geq 1$. For a rational map $\tau$ from a geometrically irreducible variety $Y$ to $\mathbb{G}_m^k$, the (PB) condition in \cite{Z1} means that for any integer $m>0$, the pullback $\mathbb{G}_m^k\times_{[m],\tau} Y$ is geometrically irreducible, where $[m]$ is the $m$-th power map. The result in \cite[Theorem 2.1]{Z1} asserts that:
\begin{align*}
&\textit{If } \tau \textit{ is a cover (that is, dominant rational map of finite degree) defined} \\& \textit{over } K_c \textit{ and satisfies the (PB) condition, there exists a finite union } W \\&\textit{of proper cosets such that if } v \in T_k \setminus W, \textit{ then } v \in \tau(Y) \textit{ and if } \tau(u)=v, \\& \textit{ then } [K_c(u):K_c]= \deg \tau.
\end{align*} We want to apply this.

Recall that $\pi_X(X^{oa}\cap \mathcal{G}^{[d]}(K_{ab})) \cap \mathbb{U}^d$ is infinite. Note that since $\pi_X$ is defined by a primitive lattice, the induced pullback map between the rings of coordinate functions is injective, and so such map is dominant, and  $\pi_X: X \rightarrow \mathbb{G}_m^d$ is a cover.

We can now factor $\pi_X$ as $\lambda_0 \circ \rho_0$ according to the second claim of \cite[Proposition 2.1]{Z1} such that $\lambda_0 :Y \rightarrow \mathbb{G}_m^d$ is an isogeny of algebraic groups and $\rho_0:X \rightarrow Y$ is a rational map satisfying the (PB) condition. Since there is a dual isogeny $\widehat{\lambda_0}: \mathbb{G}_m^d \rightarrow Y$, we see that $Y$ is isomorphic to a quotient of $\mathbb{G}_m^d$ by a finite subgroup whose elements have coordinates that are all roots of unity, and hence such subgroup is a kernel of a product of power maps $[m_1]\times...\times [m_d]: \mathbb{G}_m^d \rightarrow \mathbb{G}_m^d$ for integers $m_1,...,m_d \geq 1$. We can see that $\mathbb{G}_m^d$ modulo such kernel is isomorphic to $\mathbb{G}_m^d$. Therefore, $Y$ is isomorphic to $\mathbb{G}_m^d$ as algebraic groups. Thus, we may factor
$$
\pi_X=([m_1]\times...\times [m_d]) \circ \rho,
$$where $\rho:X \rightarrow \mathbb{G}_m^d$ is a rational map satisfying the (PB) condition. Since there are only finitely many $\pi_X$, there are only finitely many such integers $m_1,...,m_d$. Since $\pi$ and $[m_1],...,[m_d]$ are defined over $\mathbb{Q}$, $\rho$ is defined over $K_c$. We can enlarge the field $K$ by adding the finitely many $m_i$-th roots of unity, so that $\rho$ is now supposed to be defined over $K$.

Moreover, for each point $P \in X(K_{ab})$ and any Galois automorphism $\sigma$ over $K$, denoting by $P^\sigma$ the image of $P$ under $\sigma$, we have that $\rho(P^\sigma)=\rho(P)^\sigma$. 

Since $\pi_X$ is a cover, $\rho$ is also a cover. Thus the product cover
$$
\psi =\rho \times \rho: X \times X \rightarrow \mathbb{G}_m^d \times \mathbb{G}_m^d
$$ is defined over $K$ and also satisfies the (PB) condition, with $\deg \psi= (\deg \rho)^2$.

Since $X$ is defined over $K$, if $P \in X^{oa}\cap \mathcal{G}^{[d]}(K_{ab})$, then $P^\sigma \in X^{oa}\cap \mathcal{G}^{[d]}(K_{ab})$  for any Galois automorphism $\sigma$ over $K$. In view of the infinity of $\mathbb{U}^d \cap \pi_X(X(K_{ab}))$, we know that  $\mathbb{U}^d \cap \rho(X(K_{ab}))$ is also infinite. Thus considering the set $S$ of images of points of the form $(P,P^\sigma)$ under $\psi$ for any $P \in \rho^{-1}(\mathbb{U}^d) \cap X(K_{ab})$ and any Galois automorphism $\sigma$ over $K$, $S$ is an infinite set and is a torsion subset of $(\mathbb{G}_m^d)^2$, whose elements have the form $(\zeta_1,...,\zeta_d,\zeta_1^\sigma,...,\zeta_d^\sigma) \in (\mathbb{G}_m^d)^2$ with $\zeta_1,...,\zeta_d$ are roots of unity.

Applying \cite[Theorem 2.1]{Z1} to $\rho$ and $\psi$, we see that for such points $(P,P^\sigma)$ whose image under $\psi$ is outside of a finite union of proper cosets of $(\mathbb{G}_m^d)^2$ with $P \in X(K_{ab})$, we have that
$$
[K(P):K]= \deg \rho, [K(P,P^\sigma)]=\deg \psi= (\deg \rho)^2.
$$ Note that $[K(P,P^\sigma)]=[K(P):K]$ due to the normality of $K(P)/K$, and this implies that $\deg \rho =1$. Hence, $\rho$ is birational, $\pi_X \circ \rho^{-1}=[m_1]\times... \times [m_d]$ is an isogeny of $\mathbb{G}_m^d$, and thus $\pi \in \Phi^{[d]}_{X,c}$. This completes the proof.
\end{proof} The next result extends \cite[Cor. 2.3]{OSSZ1}
\begin{corollary}
Under the conditions of Theorem \ref{structural}, $$X^{oa}\cap \mathcal{G}^{[d]}(K_{ab}) \textit{ is the union of }X^{oa}\cap \mathcal{G}^{[d]}(K_{c}) \textit{ and a finite set.}$$
\end{corollary}
\begin{proof}
We first use Theorem \ref{structural}, and the fact that $$\bigcup_{\phi \in \Phi_{X,c}^{[d]}}\phi_{X}^{-1}(\mathbb{U}^{d}) \subset X(K_c)$$ by the definition of $\Phi_{X,c}^{[d]}$. Also, from the last conclusions in the proof of Theorem 3.6, we have $$\rho^{-1}=\pi_X^{-1}\circ [m_1]\times...\times[m_d]$$ for a primitive $d$-character $\pi \in \Phi_{X,c}^{[d]}$
so that $\pi_X^{-1}(\mathbb{U}^d) \subset X(K_c)$. The equation above then implies $\rho^{-1}(\mathbb{U}^d) \subset X(K_c)$.
\end{proof}
\subsection{Abelian points in curves and periodic hypersurfaces} Here we briefly look to similar problems in the context of arithmetic dynamics, in particular when curves intersect periodic hypersurfaces, following the recent dynamical version of the bounded height result for curves, , we obtain a finiteness result over the cyclotomic closure.

\begin{lemma}\cite[Theorem 4.3]{N}\label{lemdynamics}
Let $K$ be a number field. Let $f \in K[X]$ be polynomial that is  not linear conjugate to a monomial or  Chebyshev polynomial, and $\varphi=(f,...,f): (\mathbb{P}_K^1)^n \rightarrow (\mathbb{P}_K^1)^n$ be the corresponding split polynomial map. Let $C$ be an irreducible curve in $(\mathbb{P}_K^1)^n$ that is not contained in any periodic hypersurface. Assume that $C$ is non-vertical, by which we mean $C$ maps surjectively onto each factor $\mathbb{P}^1$ of $(\mathbb{P}^1)^n$. Then the points in
$$
\bigcup_{V}(C(\overline{K}) \cap V(\overline{K}))
$$ have bounded Weil heights, where $V$ ranges over all periodic hypersurfaces of $ (\mathbb{P}_K^1)^n$.
\end{lemma} The next result is basically contained in the previous one and its proof reproduced here is contained in the original proof in \cite{N}.
\begin{prop}\label{thmdynamics}
Under the conditions of \ref{lemdynamics}, let $L/K$ be an algebraic extension such that $L$ is a field that has the Bogomolov property. Then $\bigcup_{V}(C(L) \cap V(L))$ for $V$ ranging over all periodic hypersurfaces of $ (\mathbb{P}_K^1)^n$ is the union of a finite set with a set whose elements have at least one coordinate periodic for $f$ and height bounded by $\sup_{P \in \text{Per}(f)} h(P)$.
\end{prop}
\begin{proof}
It is enough to show that the claim is true for each one of the  coordinates $x_i$ of the points in the referred set. By \cite[Theorem 2.2]{N}, it suffices to show that the claim is true for points in $\bigcup(C(\overline{K}) \cap V_{ij}(\overline{K}))$, where $V_{ij}$ ranges over all periodic hypersurfaces whose equations involve $x_i$ and $x_j$ only. Thus, we may assume $n=2$ and use coordinates $x$ and $y$. By \cite[Theorem 2.2]{N}, we only need to consider the intersection with periodic curves given by an equation of the form $x=\zeta$ where $\zeta$ is $f$-periodic, or $y=g(x)$ where $g$ commutes with an iterate of $f$. Hence, we only need to deal with the second case. With this goal, given a point $(\alpha,\beta)$ in an intersection like that, we notice that the equation (31) in the proof of \cite[Theorem 4.3]{N} and the paragraph below the same equation give together that
$$
\hat{h}(\alpha)\deg g \ll \sqrt{\hat{h}(\alpha) \deg g} 
$$ for $\deg g$ sufficiently large and $\hat{h}$ a certain canonical height associated with $f$ and commensurate to $h$. This implies that $$ \sqrt{h(\alpha) \deg g} \ll 1.$$ By the Bogomolov property, this yields $\deg g \ll 1$, and by \cite[Proposition 2.3(d)]{N}, there are only finitely many such $g$'s with bounded degree, and hence only finitely many points in the intersection $C \cap \{y=g(x) \}$.
\end{proof}Below, our finiteness result follows.
\begin{theorem}
Under the conditions of \ref{lemdynamics}, the set 
$$\bigcup_{V}(C(K_c) \cap V(K_c))$$ is finite, where $V$ ranges over all periodic hypersurfaces of $ (\mathbb{P}_K^1)^n$.
\end{theorem}
\begin{proof}
Using the proof of Proposition \ref{thmdynamics} and \cite[Theorem 1.7]{OY}, we obtain the desired conclusion.
\end{proof} We can afterwards conclude the veracity of such fact for general split polynomial maps.
\begin{theorem}
Let $n \geq 2$, and let $f_1,...,f_n \in K[X]$ be polynomials of degrees at least $2$ that are not linear conjugate to monomials or to Chebyshev polynomials. Then Theorem 3.10 holds for the dynamics of the split polynomial map $\Phi=(f_1,...,f_n): (\mathbb{P}_K^1)^n \rightarrow (\mathbb{P}_K^1)^n$ and the subset of $\bigcup_V(C\cap V)$ consisting of images under finitely many split polynomial maps $$(p_1,...,p_n):(\mathbb{P}^1(K_c))^n \rightarrow (\mathbb{P}^1(K_c))^n  $$ over $K_c$. Such set is finite
\end{theorem}
\begin{proof}
By the proof of \cite[Theorem 4.16]{N}, it is enough to prove the result for $n=2$. By the same proof we reduce to the case when there are polynomials $p_1$ and $p_2$ such that $f_1 \circ p_1= p_1 \circ q$ and $f_2 \circ p_2=p_2 \circ q$, and such that for every curve $V$  that is $ \Phi$-periodic,  $(p_1,p_2)^{-1}(V)$ is $(q,q)$-preperiodic. Moreover, $q$ is also not linear conjugate to monomials or Chebyshev polynomials. Using all this, the properties of heights under the image of polynomials given by \cite[Theorem 3.11]{Silv}, and the validity of the statement for the split map $(q,q)$, we have the desired.
\end{proof}
\begin{remark}
It would be interesting to try to obtain higher dimensional generalisations of Theorem 3.11, using for example the results of \cite{GN}.
\end{remark}


\begin{thebibliography}{99} 

  \bibitem{AZ}
 F. Amoroso and U. Zannier \textit{A relative Dobrowolski lower bound over abelian extensions}, Annali della Scuola Normale Superiore di Pisa - Classe di Scienze 29.3 (2000): 711-727

  \bibitem{ADZ}
 F. Amoroso, S. David and U. Zannier \textit{On fields with Property (B)}, Proc. Amer. Math. Soc. \textbf{142} (2014), 1893--1910.

 \bibitem{BOSS} 
 A. B\' erczes, A. Ostafe, I. E. Shparlinski and J. H. Silverman, 
 \textit{Multiplicative dependence among iterated values of rational functions modulo finitely generated groups},   
 \textit{Internat. Math. Res. Notices}, \url{https://doi.org/10.1093/imrn/rnz091}. 
 
  \bibitem{BMZ1}
E. Bombieri, D. Masser and U. Zannier, \textit{Anomalous Subvarieties - Structure Theorems and Applications}, Int. Math. Res. Notices \textbf{19} (2007), 1--33.
 
 
 \bibitem{BMZ}
E. Bombieri, D. Masser and U. Zannier, \textit{Intersecting a curve with algebraic subgroups
of multiplicative groups}, Int. Math. Res. Notices \textbf{20} (1999), 1119--1140.


 
 \bibitem{CSV}
J. Chen, D. Stehl\'{e} and G. Villard, \textit{Computing an LLL-reduced Basis of the Orthogonal Lattice}, ISSAC 2018, 43rd International Symposium on Symbolic and Algebraic Computation (ISSAC 2018), Jul 2018, New York, United States.

 \bibitem{D}
H. Davenport, \textit{Intersecting a curve with algebraic subgroups
of multiplicative groups}, Proc. Roy. Soc. \textbf{272} (1963), 285--303.


\bibitem{GN} D. Ghioca and K. D. Nguyen : \textit{Dynamical anomalous subvarieties: structure and bounded height theorems}, Adv. Math. \textbf{288} (2016), 1433--1462.




 \bibitem{Hab}
P. Habegger, \textit{On the Bounded Height Conjecture}, Int. Math. Res. Notices \textbf{19} (2009), 860--886.

\bibitem{HB}  
D. R. Heath-Brown, `The density of rational points on curves and surfaces', \textit{Ann. of Math.},  \textbf{155} (2002), 553--595.

\bibitem{Hi}  
H. Hironaka, `Resolution of singularities of an algebraic variety over a field of characteristic zero. I, II', \textit{Ann. of Math.},  \textbf{79} (1964), 109--203.

\bibitem{LvdP}
J. H. Loxton and A. J. van der Poorten, \textit{Multiplicative dependence in number fields}, Acta Arith. \textbf{42} (1983), 291--302.

\bibitem{vdPL} 
J. H. Loxton and A. J. van der Poorten, \textit{Multiplicative relations in number fields}, Bull. Austral. Math. Soc. \textbf{16} (1977), 83--98.


\bibitem{M}  G. Maurin : \textit{\'{E}quations multiplicatives sur les sous-vari\'{e}t\'{e}s des tores}, Int. Math. Res. Not. IMRN, (23) (2011), pp. 5259-5366.




\bibitem{N}  K. Nguyen: \textit{Some arithmetic dynamics of diagonally split polynomial maps}, Int. Math. Res. Not. IMRN, 2015 (5) (2015), pp. 1159-1199.


\bibitem{NS}  P. Nguyen, J. Stern: \textit{Merkle-Hellman revisited: A cryptanalysis of the Qu-Vanstone cryptosystem based on group factorizations}, Advances in Cryptology-CRYPTO '97,  Volume 1294 (1997).



\bibitem{OSSZ1}
A. Ostafe, M. Sha, I. E. Shparlinski and U. Zannier, \textit{On abelian multiplicatively dependent points on a curve in a torus}, 
Q. J. Math. \textbf{69} (2018), 391--401.


\bibitem{OSSZ2}
A. Ostafe, M. Sha, I. E. Shparlinski and U. Zannier, \textit{On multiplicative dependence of values of rational functions and a generalisation 
of the Northcott theorem},  \textit{Michigan Math. J.} \textbf{68} (2019), 385--407.



\bibitem{OY}
A. Ostafe, M. Young, \textit{On algebraic integers of bounded house and preperiod- icity in polynomial semigroup dynamics},  \textit{Trans. Amer. Math. Soc.} \textbf{373} (2020), 2191--2206.

 





\bibitem{Silv} J. Silverman: \textit{The arithmetic of dynamical systems}, Springer, New York, 2007.

\bibitem{Z1} U. Zannier: \textit{Hilbert irreducibility above algebraic groups}, Duke Math. J., 153 (2010), 397-425. 


\bibitem{Z} U. Zannier: \textit{Lecture Notes on Diophantine Analysis}, Publ. Scuola Normale Superiore, Pisa, 2009.











\end{thebibliography}
\end{document}